\documentclass
{amsart}
\usepackage{graphicx}
\newtheorem{thm}{Theorem}[section]
\theoremstyle{definition}
\newtheorem{cor}[thm]{Corollary}
\newtheorem{prop}[thm]{Proposition}
\newtheorem{defn}[thm]{Definition}
\newtheorem{defns}[thm]{Definitions}
\newtheorem{lem}[thm]{Lemma}

\newtheorem{defnsnotes}[thm]{Notations and Definitions}
\newtheorem{rem}[thm]{Remark}
\newtheorem{ex}[thm]{Example}

\newtheorem{co}[thm]{Question}
\numberwithin{equation}{section}
\begin{document}
\title[The dual of $z$-submodules of modules and some of extensions]
{The dual of $z$-submodules of modules and some of extensions}
\author[F. Farshadifar, A. Molkhasi and E. Nazari]%
{F. Farshadifar*,  A. Molkhasi** and E. Nazari***}

\newcommand{\acr}{\newline\indent}
\address{\llap{*\,}(Corresponding Author) Department of Mathematics Education, Farhangian University, P.O. Box 14665-889, Tehran, Iran.}
\email{f.farshadifar@cfu.ac.ir}

\address{\llap{**\,} Department of Mathematics Education, Farhangian University, P.O. Box 14665-889, Tehran, Iran.}
\email{molkhasi@gmail.com}

\address{\llap{***\,} Department of Mathematics Education, Farhangian University, P.O. Box 14665-889, Tehran, Iran.}
\email{nazarieb2014@gmail.com}


\begin{abstract}
Let $R$ be a commutative ring with identity and $M$ be an $R$-module.
The purpose of this paper is to introduced the dual notion of $z$-submodules of $M$ and some of extensions. Moreover, we investigate some properties of these classes of modules when $M$ is a coreduced comultiplication $R$-module.
\end{abstract}
\maketitle

\section{Introduction}
\noindent
Throughout this paper, $R$ will denote a commutative ring with
identity.  We start by giving some definitions and notations needed in the sequel.
\begin{defnsnotes}\label{n1.1}
Let $M$ be an $R$-module.
\begin{itemize}
\item [(1)] A proper ideal $I$ of $R$ is called a \textit{$z$-ideal} (resp. \textit{strong $z$-ideal}) whenever any two
elements (resp. ideals) of $R$ are contained in the same set of maximal ideals and $I$ contains
one of them, then it also contains the other one \cite{MR321915}.
\item [(2)] For each $a \in R$, let $\mathfrak{P}_a$, be the intersection of all minimal prime ideals of $R$ containing $a$.
A proper ideal $I$ of $R$ is called a \textit{$z^\circ$-ideal} if for each $a \in I$
we have $\mathfrak{P}_a \subseteq I$ \cite{MR1736781}.
\item [(3)] For a submodule $N$ of $M$, let $\mathcal{M}(N)$ be the set of maximal submodules of $M$ containing
$N$. A proper submodule $N$ of $M$ is said to be a \textit{$z$-submodule} if for every $x,y \in M$, $\mathcal{M}(x)=\mathcal{M}(y)\not= \emptyset$ and $x \in N$ imply $y \in N$ \cite{F401}.
\item [(4)] The intersection of all maximal submodules of $M$ containing $N$ is said to be the \textit{Jacobson radical} of $N$ and denote by $Rad_N(M)$ \cite{MR3677368}. In case $N$ does not contained in any maximal submodule, the Jacobson radical of $N$ is defined to be $M$.  A proper submodule $N$ of $M$ is said to be a \textit{strongly $z$-submodule} if $Rad_K(M) \subseteq N$  for all submodules $K$ of $N$ with $Rad_K(M)\not=M$ \cite{F401}.
\item [(5)] A proper submodule $P$ of $M$ is said to be \textit{prime} if for any
$r \in R$ and $m \in M$ with $rm \in P$, we have $m \in P$ or $r \in (P :_R M)$ \cite{MR183747, MR498715}.
\item [(6)] A prime submodule $P$ of $M$ is a \textit{minimal prime submodule over} a submodule $N$ of $M$ if $P$ is a minimal element of the set of all
prime submodules of $M$ that contain $N$ \cite{MR2417474}.  \textit{A minimal prime submodule} of $M$ means a minimal prime submodule over the $0$ submodule of $M$. The set of all minimal prime submodules of $M$ will be denoted by $Min^p(M)$. The intersection of all minimal prime submodules of $M$ containing a submodule $K$ of $M$ is denote by $\mathfrak{P}_K$. In case $K$ does not contained in
any minimal prime submodule of $M$, $\mathfrak{P}_K$ is defined to be $M$. Also, the intersection of all minimal prime submodules of $M$ containing $x \in M$ is denote by $\mathfrak{P}_x$. In case $x$ does not contained in
any minimal prime submodule of $M$, $\mathfrak{P}_x$ is defined to be $M$.
\item [(7)] A proper submodule $N$ of $M$ is said to be a \textit{$z^\circ$-submodule} of $M$ if  $\mathfrak{P}_x \subseteq N$  for all  $x \in N$ \cite{F402}.  A proper submodule $N$ of $M$ is said to be a \textit{quasi $z^\circ$-submodule} of $M$ if $\mathfrak{P}_{aM} \subseteq N$  for all  $a \in (N:_RM)$ \cite{F403}. A proper submodule $N$ of $M$ is said to be a \textit{strong $z^\circ$-submodule} of $M$ if  $\mathfrak{P}_K \subseteq N$  for all submodules $K$ of $N$ \cite{F404}.
\item [(8)] $M$ is said to be a \emph{multiplication module} if for every submodule $N$ of $M$ there exists an ideal $I$ of $R$ such that $N=IM$ \cite{Ba81}. Also, $M$ is said to be a \emph{comultiplication module} if for every submodule $N$ of $M$ there exists an ideal $I$ of $R$ such that $N=(0:_MI)$ \cite{MR3934877}.
\item [(9)]  $M$ satisfies the \emph{double annihilator
conditions} (DAC for short) if, for each ideal $I$ of $R$,
we have $I=Ann_R((0:_MI))$. Also, $M$ is a \emph{strong comultiplication module} if $M$ is
a comultiplication $R$-module and satisfies the DAC conditions \cite{MR3934877}.
\item [(10)] A proper submodule $N$ of $M$ is said to be \emph{completely irreducible} if $N=\bigcap _
{i \in I}N_i$, where $ \{ N_i \}_{i \in I}$ is a family of
submodules of $M$, implies that $N=N_i$ for some $i \in I$. By \cite{FHo06},
every submodule of $M$ is an intersection of completely irreducible submodules of $M$. Thus the intersection
of all completely irreducible submodules of $M$ is zero \cite{FHo06}.
\item [(11)] $M$ is said to be \textit{reduced} if the intersection of all prime submodules of $M$
is equal to zero \cite{MR2839935}. Also, $M$ is said to be \emph{coreduced module} if $(L:_Mr)=M$ implies that $L+(0:_Mr)=M$, where $r \in R$ and $L$ is a completely irreducible submodule of $M$ \cite{MR3755273}.
\end{itemize}
\end{defnsnotes}
In \cite{F401, F402, F404, F403}, the notions of $z$-submodules, strong $z$-submodules, $z^\circ$-submodules,  strong $z^\circ$-submodules, and quasi $z^\circ$-submodules of $M$ were introduced and investigated some of their properties when $M$ is a reduced multiplication $R$-module. When a concept is defined in algebra, the question naturally arises, what is the dual of this concept? The purpose of this paper is to introduced the dual of the concepts $z$-submodules, strong $z$-submodules, $z^\circ$-submodules,  strong $z^\circ$-submodules, and quasi $z^\circ$-submodules of $M$. Moreover, we investigate some properties of these classes of modules when $M$ is a coreduced comultiplication $R$-module.

\section{Dual $z$-submodules and strong $z$-submodules}
Let $M$ be an $R$-module.  For a submodule $N$ of $M$, let $\mathcal{M^*}(N)$ be the set of minimal submodules of $M$ contained in $N$ and the set of all minimal submodules of $M$ will be denoted by $Min(M)$. If $N$ is a submodule of $M$, define
 $V^*(N) = \{S \in Min(M) : S \subseteq N\}$. The sum of all minimal submodules of $M$ contained in a submodule $K$ of $M$ is denote by $Soc_K$.  In case $M$ does not contain any minimal submodule which is contained in $K$, then $Soc_K$ is defined to be $(0)$.
\begin{defns}\label{d00.1}
Let $M$ be an $R$-module.
\begin{itemize}
\item [(a)] We say that a non-zero submodule $N$ of $M$ is a \textit{dual $z$-submodule} or  \textit{$dz$-submodule} if $N \subseteq Soc_L$  for all completely irreducible submodules $L$ of $M$ with $N\subseteq L$.
\item [(b)] We say that a non-zero submodule $N$ of $M$ is a \textit{dual strong $z$-submodule} or  \textit{$dsz$-submodule} if $N \subseteq Soc_K$  for all submodules $K$ of $M$ with $N\subseteq K$.
\end{itemize}
\end{defns}

If $N$ is a $dsz$-submodule of an $R$-module $M$, then $N$ is a $dz$-submodule of $M$.
It is natural to ask the following question:
\begin{co}
Let $M$ be an $R$-module. Is every $dz$-submodule of $M$ a $dsz$-submodule of $M$?
\end{co}

An $R$-module $M$ is said to be \emph{cocyclic} if
$Soc_M$ is a large and simple submodule of $M$ \cite{Y98}. A submodule $L$ of $M$ is a completely irreducible
 submodule of $M$ if and only if $M/L$ is a cocyclic $R$-module \cite{FHo06}.
\begin{ex}\label{e0.12}
\begin{itemize}
\item [(a)] Since the  $\Bbb Z$-module $\Bbb Z$ has no minimal submodule, we have every non-zero submodule of the $\Bbb Z$-module $\Bbb Z$ is not a $dz$-submodule.
\item [(b)] The only $dz$-submodule of a cocyclic $R$-module is the minimal submodule of it. In particular, $\langle 1/p+\Bbb Z\rangle$ is the only $dz$-submodule of the $\Bbb Z$-module $\Bbb Z_{p^\infty}$.
\end{itemize}
\end{ex}

\begin{prop}\label{l00.2}
Let $N$ be a non-zero submodule of an $R$-module $M$. Then we have the following.
\begin{itemize}
\item [(a)] $N$ is a $dz$-submodule of $M$ if and only if when $L$ is a completely irreducible submodule of $M$ with $N\subseteq L$, $H$ a submodule of $M$, and $Soc_L \subseteq Soc_H$, then $N \subseteq H$.
\item [(b)] $N$ is a $dsz$-submodule of $M$ if and only if when $K$ is a submodule of $M$ with $N\subseteq K$, $H$ a submodule of $M$, and $Soc_K \subseteq Soc_H$, then $N \subseteq H$.
\end{itemize}
\end{prop}
\begin{proof}
(a) First suppose that $N$ is a $dz$-submodule of $M$.
Let $L$ be a completely irreducible submodule of $M$ with $N\subseteq L$, $H$ a submodule of $M$, and $Soc_L \subseteq Soc_H$. Then by assumption,
$$
N \subseteq Soc_L \subseteq Soc_H\subseteq H.
$$
For the converse, let $L$ be a completely irreducible submodule of $N$ with $N\subseteq L$. Then
$
Soc_{Soc_L}=Soc_L$, implies that $N \subseteq Soc_L$, by assumption.

(b) This is similar to the proof of part (a).
\end{proof}

\begin{rem}\label{r0.3}
Let $M$ be an $R$-module.
If $N$ is a $dz$-submodule (resp. $dsz$-submodule) of $M$, then for each completely irreducible submodule $L$ (resp. submodule $K$) of $M$ with $N\subseteq L$ (resp. $N\subseteq K$)
we have  $Soc_L\not=0$ (resp. $Soc_K\not=0$), i.e. $L$ (resp. $K$) contains at least one minimal submodule of $M$.
Clearly, every minimal submodule of $M$ is a $dsz$-submodule of $M$. Also, the family of $dsz$-submodules of $M$ is closed under summation. Therefore, if  $Soc_M\not=0$, then $Soc_M$ is a $dsz$-submodule of $M$ and it contains every  $dsz$-submodule of $M$.
\end{rem}

\begin{prop}\label{p0.4}
Let $N$ be a non-zero submodule of an $R$-module $M$. Then $N$ as an $R$-submodule is a  $dz$-submodule (resp. $dsz$-submodule) if and only if as an $R/Ann_R(M)$-submodule is a $dz$-submodule (resp.  $dsz$-submodule).
\end{prop}
\begin{proof}
This is straightforward.
\end{proof}

The intersection of all maximal ideals of $R$ containing an ideal $I$ of $R$ is denote by $\mathfrak{M}_I$.
\begin{thm}\label{t0.5}
Let $M$ be an $R$-module. Then $Soc_{(0:_MI)}\subseteq (0:_M\mathfrak{M}_I)$ for each ideal $I$ of $R$. The reverse inclusion holds when $M$ is a faithful finitely generated comultiplication $R$-module.
\end{thm}
\begin{proof}
Let $I$ be an ideal of $R$ and $T$ be a minimal submodule of $M$ such that $T \subseteq (0:_MI)$. Then $I \subseteq Ann_R(T)$. As $Ann_R(T)$ is a maximal ideal of $R$, we have $\mathfrak{M}_I \subseteq Ann_R(T)$. Therefore,  $T\subseteq (0:_MAnn_R(T)) \subseteq (0:_M\mathfrak{M}_I)$. Hence,
$Soc_{(0:_MI)}\subseteq (0:_M\mathfrak{M}_I)$ for each ideal $I$ of $R$. For the reverse inclusion, let $\mathfrak{M}_I=\bigcap \mathfrak{m}_i$, where
$\mathfrak{m}_i$ ranges over all maximal ideals of $R$ that contains $I$.
As $M$ is a faithful finitely generated comultiplication $R$-module, we have $\sum (0:_M\mathfrak{m}_i)=(0:_M\bigcap{\mathfrak{m}_i})$ by using \cite[Proposition 2.1(b)]{F407}. Thus
$$
\sum_{(0:_M\mathfrak{m}_i)\not=0}(0:_M\mathfrak{m}_i)=\sum (0:_M\mathfrak{m}_i)=(0:_M\mathfrak{M}_I)\subseteq (0:_MI).
$$
It follows that $(0:_M\mathfrak{M}_I)\subseteq Soc_{(0:_MI)}$ since by \cite[Theorem 7 (f)]{MR3934877}, $(0:_M\mathfrak{m}_i)\not=0$ is a minimal submodule of $M$.
\end{proof}

\begin{cor}\label{c0.6}
Let $N$ be a $dsz$-submodule of an $R$-module $M$. Then $Ann_R(N)$ is a strongly $z$-ideal of $R$. The converse holds when $M$ is a faithful finitely generated comultiplication $R$-module.
\end{cor}
\begin{proof}
First let $N$ be a $dsz$-submodule of $M$ and $I$ be an ideal of $R$ such that $I\subseteq Ann_R(N)$. Then $N \subseteq (0:_MI)$ and so by assumption, $N \subseteq Soc_{(0:_MI)}$. Thus  $N \subseteq  (0:_M\mathfrak{M}_I)$ by Theorem \ref{t0.5}. This implies that
$\mathfrak{M}_I\subseteq Ann_R(N)$. Hence $Ann_R(N)$ is a strong $z$-ideal of $R$. For converse, let $K$ be a submodule of $M$ such that $N\subseteq K$. Thus $Ann_R(K)\subseteq Ann_R(N)$. Therefore, $\mathfrak{M}_{Ann_R(K)}\subseteq Ann_R(N)$ by assumption.  Now we have $Soc_{(0:_MAnn_R(K))}=(0:_M\mathfrak{M}_{Ann_R(K)})$ by Theorem \ref{t0.5}. It follows that
$$
N=(0:_MAnn_R(N))\subseteq Soc_{(0:_MAnn_R(K))}=Soc_{K}.
$$
\end{proof}

\begin{cor}\label{c0.7}
\begin{itemize}
\item [(a)] Let $M$ be a finitely generated strong comultiplication $R$-module and $I$ be a strong  $z$-ideal of $R$. Then $(0:_MI)$ is a $dsz$-submodule of $M$.
\item [(b)] Let $N$ be a $dsz$-submodule of an $R$-module $M$. Then $((0:_MAnn_R(K):_RN)$ is a $z$-ideal of $R$ for each submodule $K$ of $M$. In particular, if $Soc_M=M$, then $((0:_MAnn_R(K):_RM)$ is a $z$-ideal of $R$ for each submodule $K$ of $M$.
\item [(c)] Let $N$ be a $dsz$-submodule of a comultiplication $R$-module $M$. Then $(K:_RN)$ is a $z$-ideal of $R$ for each submodule $K$ of $M$. In particular, if $Soc_M=M$, then $(K:_RM)$ is a $z$-ideal of $R$ for each submodule $K$ of $M$.
\end{itemize}
 \end{cor}
\begin{proof}
(a) Since $M$ is a strong comultiplication $R$-module, $I=Ann_R((0:_MI))$. Now the result follows from Corollary  \ref{c0.6}.

(b) As $N$ is a $dsz$-submodule, $Ann_R(N)$ is a strong $z$-ideal of $R$ by Corollary  \ref{c0.6}. Let $K$ be a submodule of $M$. Then by \cite[Proposition 1.3]{MR321915}, $(Ann_R(K):_RAnn_R(N))$ is a $z$-ideal of $R$. Now
$$
((0:_MAnn_R(K)):_RN)=(Ann_R(K):_RAnn_R(N))
$$
implies that $((0:_MAnn_R(K)):_RN)$ is a $z$-ideal of $R$. Now the last assertion follows from the fact that $Soc_M$ is a  $dsz$-submodule of $M$ by Remark \ref{r0.3}.

(c) As $M$ is a comultiplication $R$-module, $K=(0:_MAnn_R(K))$. Now the result follows from part (b)
\end{proof}

Let $M$ be an $R$-module. A non-zero submodule $S$ of $M$ is said to be \emph{second} if for each $a \in R$, the homomorphism $ S \stackrel {a} \rightarrow S$ is either surjective or zero. In this case, $Ann_R(S)$ is a prime ideal of $R$ \cite{MR1879449}.  A second submodule $S$ of $M$ is said to be a \emph{maximal second submodule} of $M$, if there does not exist a second submodule $N$ of $M$ such that $S \subset N \subset M$ \cite{MR3073398}.
The set of all maximal second submodules of $M$ will be denoted by $Max^s(M)$.
\begin{prop}\label{p0.10}
Let $M$ be a faithful finitely generated comultiplication $R$-module.
If $S$ is maximal in the class of second submodules contained a $sdz$-submodule $N$, then
$S$ is a strongly $dsz$-submodule.
\end{prop}
\begin{proof}
This follows from Corollary \ref{c0.6}, \cite[Proposition 2.1]{F407}, and \cite[Theorem 1.1]{MR321915}.
\end{proof}

\begin{cor}\label{c0.11}
The maximal second submodules of a faithful finitely generated comultiplication $R$-module $M$ are $dsz$-submodules.
\end{cor}

\begin{rem}\label{r2.1}\cite[Remark 2.1]{MR2821719}
Let $N$ and $K$ be two submodules of an $R$-module $M$. To prove $N\subseteq K$, it is enough to show that if $L$ is a completely irreducible submodule of $M$ such that $K\subseteq L$, then $N\subseteq L$.
\end{rem}
Let $N$ and $K$ be two submodules of an $R$-module $M$. The \emph{coproduct} of $N$ and $K$ is defined by $(0:_MAnn_R(N)Ann_R(K))$  and denoted by $C(NK)$ \cite{MR3934877}.
\begin{prop}\label{p0.8}
Let $N$ be a non-zero submodule of a comultiplication $R$-module $M$. Then the following assertions are equivalent.
\begin{itemize}
\item [(a)] $N$ is a $dsz$-submodule of $M$.
\item [(b)]
If $K$ is a submodule of $M$ with $N \subseteq K$, $H$ a submodule of $M$, and $\mathcal{M^*}(K) \subseteq \mathcal{M^*}(H)$, then $N \subseteq H$.
\item [(c)]
If $K$ is a submodule of $M$ with $N \subseteq K$, $H$ a submodule of $M$, and $\mathcal{M^*}(H)= \mathcal{M^*}(K)$, then $N \subseteq H$.
\end{itemize}
\end{prop}
\begin{proof}
$(a)\Rightarrow (b)$
Let $K$ be a submodule of $M$ with $N \subseteq K$, $H$ a submodule of $M$, and $\mathcal{M^*}(K) \subseteq \mathcal{M^*}(H)$. Then $Soc_K \subseteq Soc_H$ and by part (a),   $N \subseteq H$.

$(b)\Rightarrow (c)$ This is clear.

$(c)\Rightarrow (a)$
Let $K$ be a submodule of $M$ with $N \subseteq K$. We have to prove that $N \subseteq Soc_K$.
Let  $L$ be a completely irreducible submodule of $M$ such that $Soc_K\subseteq L$. Then $\mathcal{M^*}(K)\subseteq \mathcal{M^*}(L)$. We show that  $\mathcal{M^*}(C(KL))=\mathcal{M^*}(K)$. Suppose that exists a minimal submodule $T$ of $M$ contained in  $C(KL)$ but $T\not \subseteq L$. Then as $M$ is a comultiplication $R$-module we conclude that
$T \subseteq K$, so $T \in \mathcal{M^*}(K)\subseteq \mathcal{M^*}(L)$. Thus $T \subseteq L$, which is impossible. The reverse inclusion is clear. We conclude that
$\mathcal{M^*}(C(KL)) =\mathcal{M^*}(L)$ with $N \subseteq C(LK)$. By part (c),  $ N\subseteq L$. Now by Remark \ref{r2.1}, $N \subseteq Soc_K$,  as needed.
\end{proof}

\begin{cor}\label{c0.9}
Any $dsz$-submodule $N$ of a comultiplication $R$-module $M$ is equal to the sum of all the minimal submodules
contaned in it.
\end{cor}
\begin{proof}
Let $ Soc_N \subseteq L$ for some completely irreducible submodule $L$ of $M$. Then $\mathcal{M^*}(N) \subseteq \mathcal{M^*}(L)$. Thus  by Proposition \ref{p0.8}, $N\subseteq L$. This implies that $N \subseteq Soc_N$ by Remark \ref{r2.1}. The reverse inclusion is clear.\end{proof}

\begin{prop}\label{p0.13}
Let $N$ be a  $dz$-submodule (resp. $dsz$-submodule) of an $R$-module $M$. Then for each $r \in R$, $rN$ is a $dz$-submodule (resp. $dsz$-submodule) of $M$.
\end{prop}
\begin{proof}
Suppose that $r \in R$ and $L$ is a completely irreducible submodule of $M$ such that $rN \subseteq L$. Then $N \subseteq (L:_Mr)$. By \cite[Lemma 2.1]{MR3588217}, $(L:_Mr)$ is a completely irreducible submodule of $M$. So by assumption, $N \subseteq Soc_{(L:_Mr)}$. But one can see that, $Soc_{(L:_Mr)} \subseteq (Soc_L:_Mr)$. Therefore, $rN \subseteq Soc_L$. The proof for $dsz$-submodule is similar.
\end{proof}

\section{The dual of $z^\circ$-submodules and strong $z^\circ$-submodules}
The sum of all maximal second submodules of an $R$-module $M$ contained in a submodule $K$ of $M$ is denote by $\mathfrak{S}_K$.  In case $M$ does not contain any maximal second submodule which is contained in $K$, then $\mathfrak{S}_K$ is defined to be $(0)$. If $N$ is a submodule of $M$, define
 $V^s(N) = \{S \in Max^s(M) : S \subseteq N\}$.

\begin{defns}\label{d7.1}
Let $M$ be an $R$-module.
\begin{itemize}
\item [(a)] We say that a non-zero submodule $N$ of $M$ is a \textit{dual $z^\circ$-submodule}  or  \textit{$dz^\circ$-submodule} of $M$ if $N \subseteq \mathfrak{S}_{L}$  for all completely irreducible submodules $L$ of $M$ with $N \subseteq L$.
\item [(b)] We say that a non-zero submodule $N$ of $M$ is a \textit{dual strong $z^\circ$-submodule} or \textit{$dsz^\circ$-submodule} of $M$ if  $N\subseteq \mathfrak{S}_K$  for all submodules $K$ of $M$ with $N\subseteq K$. Also, we say that a non-zero ideal $I$ of $R$ is a \textit{$dsz^\circ$-ideal} if $I$ is a dual strong $z^\circ$-submodule of an $R$-module $R$.
\end{itemize}
\end{defns}

If $N$ is a $dsz^\circ$-submodule of $M$, then $N$ is a $dz^\circ$-submodule of $M$.
It is natural to ask the following question:
\begin{co}
Let $M$ be an $R$-module. Is every $dz^\circ$-submodule of $M$ a $dsz^\circ$-submodule of $M$?
\end{co}

\begin{rem}\label{r7.2}
Let $M$ be an $R$-module.
If $N$ is a $dz^\circ$-submodule of $M$, then for
each completely irreducible submodule $L$ of $M$ with $N \subseteq L$ we have  $\mathfrak{S}_L\not=0$, i.e. $L$ contains at least a maximal second submodule of $M$. Also, if $N$ is a $dsz^\circ$-submodule of $M$, then
for every submodule $K$ of $M$ with $N \subseteq K$ we have  $\mathfrak{S}_K\not=0$. Clearly, every maximal second submodule of $M$ is a $dsz^\circ$-submodule of $M$. Also, the family of $dsz^\circ$-submodule (resp.  $dz^\circ$-submodules) of $M$ is closed under intersection. Therefore, if $\mathfrak{S}_M\not=0$, then $\mathfrak{S}_M$ is a $dsz^\circ$-submodule of $M$ and it is contains every $dsz^\circ$-submodule of $M$.
\end{rem}

\begin{lem}\label{l7.1}
Let $M$ be an $R$-module.
A submodule $N$ of $M$ is a $dsz^\circ$-submodule if and only if $N=\cap_{K\in \Lambda} \mathfrak{S}_K$, where $\Lambda$ is the collection of all submodules of $M$ with $N \subseteq K$.
\end{lem}
\begin{proof}
This is straightforward.
\end{proof}

\begin{prop}\label{p7.1}
Let $N$ be a non-zero submodule of an $R$-module $M$. Then $N$ as an $R$-submodule is a $dz^\circ$-submodule (resp. $dsz^\circ$-submodule) if and only if as an $R/Ann_R(M)$-submodule is a $dz^\circ$-submodule (resp. $dsz^\circ$-submodule).
\end{prop}
\begin{proof}
This is straightforward.
\end{proof}

The intersection of all minimal prime ideals of $R$ containing an ideal $I$ of $R$ is denote by $\mathfrak{P}_I$.
\begin{rem}\label{rrr7.1}
Let $M$ be a faithful finitely generated comultiplication $R$-module. Then $(0:_M\mathfrak{P}_I)=\mathfrak{S}_{(0:_MI)}$ for each ideal $I$ of $R$ \cite[Theorem 2.3]{F407}.
\end{rem}

\begin{thm}\label{t7.3}
Let $M$ be a faithful finitely generated comultiplication $R$-module. Then $N$ is a $dsz^\circ$-submodule of $M$ if and only if $Ann_R(N)$ is a strong $z^\circ$-ideal of $R$.
\end{thm}
\begin{proof}
Let $I$ be an ideal of $R$ such that $I\subseteq Ann_R(N)$. Then $N \subseteq (0:_MI)$ and so by assumption, $N \subseteq \mathfrak{S}_{(0:_MI)}$. Hence  $N \subseteq (0:_M\mathfrak{P}_I)$ by Remark \ref{rrr7.1}. It follows that
$\mathfrak{P}_I \subseteq Ann_R(N)$ and so $Ann_R(N)$ is a strong $z^\circ$-ideal of $R$. For converse, let $K$ be a submodule of $M$ such that $N\subseteq K$. Then $Ann_R(K)\subseteq Ann_R(N)$. Now by assumption, $\mathfrak{P}_{Ann_R(K)}\subseteq Ann_R(N)$. By Remark \ref{rrr7.1}, $\mathfrak{S}_{(0:_MAnn_R(K))}=(0:_M\mathfrak{P}_{Ann_R(K)})$. Therefore,
$$
N=(0:_MAnn_R(N))\subseteq \mathfrak{S}_{(0:_MAnn_R(K))}=\mathfrak{S}_{K}.
$$
\end{proof}

\begin{cor}\label{c7.3}
Let $M$ be a finitely generated strong comultiplication $R$-module.  If $I$ is a strong $z^\circ$-ideal of $R$, then $(0:_MI)$ is a $dsz^\circ$-submodule of $M$.
\end{cor}
\begin{proof}
As $M$ is a strong comultiplication $R$-module, $I=Ann_R((0:_MI))$. Now the result follows from Theorem \ref{t7.3}.
\end{proof}

\begin{lem}\label{ll0701.14}
Let $M$ ba an $R$-module. Then $\mathfrak{S}_{(K:_Mr)} \subseteq (\mathfrak{S}_{K}:_Mr)$ for each submodule $K$ of $M$ and $r \in R$.
\end{lem}
\begin{proof}
Let $K$ be a submodule of $M$ and $r \in R$. Suppose that $S$ is a maximal second submodule of $M$ such that $S \subseteq (K:_Mr)$. Then $rS\subseteq K$. As $S$ is second, $rS=0$ or $rS=S$. This implies that  $rS\subseteq \mathfrak{S}_{K}$ and so $S\subseteq (\mathfrak{S}_{K}:_Mr)$. Therefore, $\mathfrak{S}_{(K:_Mr)} \subseteq (\mathfrak{S}_{K}:_Mr)$.
\end{proof}

\begin{prop}\label{pp0701.14}
Let $N$ be a $dz^\circ$-submodule (resp.  $dsz^\circ$-submodule) of an $R$-module $M$. Then
for each $r \in R$, $rN$ is a $dz^\circ$-submodule (rasp. $dsz^\circ$-submodule) of $M$.
\end{prop}
\begin{proof}
Suppose that $r \in R$ and $rN \subseteq L$ for some completely irreducible submodule $L$ of $M$. Then $N\subseteq (L:_Mr)$. By \cite[Lemma 2.1]{MR3588217}, $(L:_Mr)$ is a completely irreducible submodule of $M$. Now by assumption, $N \subseteq \mathfrak{S}_{(L:_Mr)}$. By Lemma \ref{ll0701.14}, $\mathfrak{S}_{(L:_Mr)}\subseteq (\mathfrak{S}_{L}:_Mr)$. Thus $N\subseteq (\mathfrak{S}_{L}:_Mr)$ and so $rN \subseteq \mathfrak{S}_{L}$. The proof for $dsz^\circ$-submodule is similar.
\end{proof}

In the following theorem, we characterize the $dsz^\circ$-submodules of an $R$-module $M$.
\begin{thm}\label{t7.6}
Let $N$ be a non-zero submodule of an $R$-module $M$. Then the following are equivalent:
\begin{itemize}
\item [(a)] $N$ is a $dsz^\circ$-submodule of $M$;
\item [(b)] For submodules $K$, $H$ of $M$, $\mathfrak{S}_K=\mathfrak{S}_H$ and $N \subseteq K$ imply that $N \subseteq H$;
\item [(c)] For submodules $K$, $H$ of $M$, $V^s(K)=V^s(H)$ and $N \subseteq K$ imply that $N \subseteq H$;
\item [(d)] If $K$ is a submodule of $M$ with $N \subseteq K$, $H$ a submodule of $M$, and $V^s(K) \subseteq V^s(H)$, then $N \subseteq H$.
\end{itemize}
\end{thm}
\begin{proof}
$(a)\Rightarrow (b)$
 Let  $K$, $H$ be two submodules of $M$ such that $\mathfrak{S}_K=\mathfrak{S}_H$ and $N \subseteq K$. By part (a), $N \subseteq \mathfrak{S}_K$. Thus $N \subseteq  \mathfrak{S}_H \subseteq H$.

 $(b)\Rightarrow (c)$
 Let  $K$, $H$ be submodules of $M$ such that $V^s(K)=V^s(H)$ and $N \subseteq K$. Then  $\mathfrak{S}_K=\mathfrak{S}_H$. Thus by part (b), $N \subseteq H$.

$(c)\Rightarrow (a)$
Let $K$ be a submodule of $M$ with $N \subseteq K$. One can see that $V^s(K)=V^s(\mathfrak{S}_K)$. Now by part (c), $N \subseteq \mathfrak{S}_K$.

$(a)\Rightarrow (d)$
Let $K$ be a submodule of $M$ with $N \subseteq K$, $H$ a submodule of $M$, and $V^s(K) \subseteq V^s(H)$. Then
$\mathfrak{S}_K \subseteq \mathfrak{S}_H$. Hence by part (a),   $N \subseteq H$.

$(d)\Rightarrow (c)$ This is clear.
\end{proof}

The following corollary gives some characterizations for $dsz^\circ$-submodules of a Noetherian coreduced comultiplication $R$-module $M$.
\begin{cor}\label{c9.6}
Let $N$ be a non-zero submodule of a Noetherian coreduced comultiplication $R$-module $M$. Then for submodules $K$ and $H$ of $M$ the following are equivalent:
\begin{itemize}
\item [(a)] $N$ is a $dsz^\circ$-submodule of $M$;
\item [(b)] $N$ is a sum of maximal second submodules of $M$;
\item [(c)] $Ann_R(N)$ is a strong $z^\circ$-ideal of $R$;
\item [(d)]  $\mathfrak{S}_K=\mathfrak{S}_H$ and $N \subseteq  K$ imply that $N  \subseteq H$;
\item [(e)] $V^s(K)=V^s(H)$ and $N \subseteq K$ imply that $N  \subseteq H$;
\item [(f)] $V^s(K) \subseteq V^s(H)$ and $N \subseteq K$ imply that $N \subseteq H$.
\item [(g)]  $(K:_RM)=(H:_RM)$ and $N  \subseteq  K$ imply that $N  \subseteq H$;
\item [(h)] For submodule $K$ of $M$, $N \subseteq K$ implies that $N \subseteq (K:_RM)M$.
\end{itemize}
\end{cor}
\begin{proof}
Follows from Theorem \ref{t7.3}, Theorem \ref{t7.6}, and \cite[Theorem 2.24]{F407}.
\end{proof}

\begin{ex}\label{t97.9}
Let $M$ be a Noetherian coreduced multiplication and comultiplication $R$-module. Then every non-zero submodule of $M$ is a $dsz^\circ$-submodule of $M$. In particular, if $n$ is square-free, every non-zero submodule of the $\Bbb Z$-module $\Bbb Z_n$ is a $dsz^\circ$-submodule.
\end{ex}

For a submodule $N$ of an $R$-module $M$, the \emph{second radical} (or \emph{second socle}) of $N$ is defined  as the sum of all second submodules of $M$, contained in $N$, and it is denoted by $sec(N)$ (or $soc(N)$). In case $N$ does not contain any second submodule, the second radical of $N$ is defined to be $(0)$  \cite{MR3085034, MR3073398}.
\begin{thm}\label{t97.9}
Let $M$ be a Noetherian coreduced comultiplication $R$-module and $N$ be a $dsz^\circ$-submodule of $M$. Then every
maximal second submodule of $N$ is a $dsz^\circ$-submodule of $M$. In particular, $sec(N)$ is a $dsz^\circ$-submodule of $M$.
\end{thm}
\begin{proof}
Let $S$ be a maximal second submodule of $N$. Assume that $(K:_RM)=(H:_RM)$, where $K, H$ are submodules of $M$ with
$S \subseteq K$. By \cite[Prop. 2.14 and Cor. 2.15]{F407}, there exists $c \in (K:_RM) \setminus Ann_R(S)$. Thus $cN \subseteq K$ and $c \not \in Ann_R(S)$. Clearly, $((K:_Mc):_RM)=((H:_Mc):_RM)$. As $N$ is a $dsz^\circ$-submodule of $M$, we have
$S \subseteq N \subseteq (H:_Mc)$. As $c \not \in Ann_R(S)$ and $S$ is a second submodule, $S \subseteq H$ as needed. Now, the last assertion is clear.
\end{proof}

\begin{cor}\label{c97.10}
Let $M$ be a Noetherian coreduced comultiplication $R$-module and $f: N\hookrightarrow M$ be the natural inclusion, where $N$ is a $dsz^\circ$-submodule of $M$. If $K$ is a $dsz^\circ$-submodule of $N$, then $f(N)$ is $dsz^\circ$-submodule of $M$.
\end{cor}

\begin{thm}\label{p97.13}
Let $N_i$ for $1\leq i\leq n$ be a non-zero submodule of a Noetherian coreduced comultiplication $R$-module $M$ such that for each $i\not=j$, $Ann_R(N_i)$ and $Ann_R(N_j)$ are co-prime ideals of $R$. Then $\sum^n_{i=1}N_i$ is a $dsz^\circ$-submodule of $M$ if and only if each $N_j$ for $1\leq j\leq n$  is a $dsz^\circ$-submodule of $M$.
\end{thm}
\begin{proof}
Assume that $\sum^n_{i=1}N_i$ is a $dsz^\circ$-submodule of $M$ and $1\leq j\leq n$. We show that $N_j$ is a $dsz^\circ$-submodule of $M$. So assume that $(K:_RM)=(H:_RM)$ for some submodules $K, H$ of $M$ with $N_j \subseteq K$. Since for each $i\not=j$,  $Ann_R(N_i)$ and $Ann_R(N_j)$ are co-prime ideals of $R$, $\bigcap^n_{i=1, i \not=j}Ann_R(N_i)$ and $Ann_R(N_j)$ are co-prime ideals of $R$. Thus $1=a+b$ for some $a \in \bigcap^n_{i=1, i \not=j}Ann_R(N_i)$ and $b \in Ann_R(N_j)$. So, $H=(H:_Ma)\cap (H:_Mb)$
and $(K:_Ma):_RM)=(H:_Ma):_RM)$. Now, we have $\sum^n_{i=1}N_i\subseteq (K:_Ma)$. Thus $\sum^n_{i=1}N_i$ is a $dsz^\circ$-submodule of $M$ which implies that $N_j\subseteq \sum^n_{i=1}N_i \subseteq (H:_Ma)$. Now since $N_j \subseteq (H:_Mb)$, we have  $N_j \subseteq H$ and we are done. The converse is clear.
\end{proof}

Let $R_i$ be a commutative ring with identity, $M_i$ be an $R_i$-module for $i = 1, 2$. Assume that
$M = M_1\times M_2$ and $R = R_1\times R_2$. Then $M$ is clearly
an $R$-module with componentwise addition and scalar multiplication. Also,
each submodule $N$ of $M$ is of the form $N = N_1\times N_2$, where $N_i$ is a
submodule of $M_i$ for $i = 1, 2$.

\begin{lem}\label{l444.6}
Let  $M = M_1\times M_2$ be an $R = R_1\times R_2$-module, where $M_i$ is an $R_i$-module for $i = 1, 2$. If $N = N_1\times N_2$ is a submodule of $M$, then $V^s(N_1\times N_2)=V^s(N_1\times 0) \cup V^s(0\times N_2)$.
\end{lem}
\begin{proof}
By \cite[Lemma 2.23]{HF17}, $S$ is a second submodule of $M$ if and only if $S=S_1\times 0$ or $S=S_2 \times 0$, where $S_1$ is a second submodule of $ M_1$ and $S_2$ is a second submodule of $M_2$. This implies that $S \in V^s(N_1\times N_2)$ if and only if
$S \in V^s(N_1\times 0)$ or $S \in V^s(0\times N_2)$, as needed.
\end{proof}

\begin{thm}\label{t2.6}
Let  $M = M_1\times M_2$ be an $R = R_1\times R_2$-module, where $M_i$ is an $R_i$-module for $i = 1, 2$. Let $N = N_1\times N_2$ be a submodule of $M$. Then the following are equivalent:
\begin{itemize}
\item [(a)] $N$ is a $dsz^\circ$-submodule of $M$;
\item [(b)] $N_i$ is a $dsz^\circ$-submodule of $M_i$ for each $i = 1, 2$.
\end{itemize}
\end{thm}
\begin{proof}
$(a)\Rightarrow (b)$
Let for submodules $K_1$, $H_1$ of $M_1$, $V^s(K_1)=V^s(H_1)$ and $N_1\subseteq K_1$.
Set $K=K_1\times  N_2$ and $H=H_1\times N_2$. Then by using  Lemma \ref{l444.6}, $V^s(K)=V^s(H)$ and $N \subseteq K$. Thus by part (a) and Theorem \ref{t7.6}, $N \subseteq H$. This implies that $N_i \subseteq H_1$ and so $N_1$ is a $dsz^\circ$-submodule of $M_1$ by Theorem \ref{t7.6}. Similarly, $N_2$ is a $dsz^\circ$-submodule of $M_2$

$(b)\Rightarrow (a)$
Let for submodules $K=K_1\times K_2$, $H=H_1\times H_2$ of $M$, $V^s(K)=V^s(H)$ and $N\subseteq K$.
Then $V^s(K_i)=V^s(H_i)$ and $N_i\subseteq K_i$ by Lemma \ref{l444.6} for $i = 1, 2$.  Now by part (b),  $N_i \subseteq H$  for $i = 1, 2$. It follows that  $N \subseteq H$, as needed.
\end{proof}
\section{The dual of quasi $z^\circ$-submodules}

\begin{defn}\label{d1.1}
 We say that a non-zero submodule $N$ of an $R$-module $M$ is a \textit{dual quasi $z^\circ$-submodule}  or  \textit{$dqz^\circ$-submodule} of $M$ if $N \subseteq \mathfrak{S}_{(0:_Ma)}$  for all  $a \in Ann_R(N)$.
\end{defn}

\begin{rem}\label{r1.1}
Let $M$ be an $R$-module.
If $N$ is a $dqz^\circ$-submodule of $M$, then for
each $a \in Ann_R(N)$ we have  $\mathfrak{S}_{(0:_Ma)}\not=0$, i.e. $(0:_Ma)$ contains at least a maximal second submodule of $M$.   Clearly, every maximal second submodule of $M$ is a $dqz^\circ$-submodule of $M$. Also, the family of
$dqz^\circ$-submodules of $M$ is closed under summation. Therefore, if  $\mathfrak{S}_{M}\not=0$, then $\mathfrak{S}_{M}$ is a $dqz^\circ$-submodule of $M$ and it is contains every $dqz^\circ$-submodule of $M$.
\end{rem}

\begin{rem}\label{r97.2}
Clearly, if $N$ is a $dsz^\circ$-submodule of $M$, then $N$ is a $dqz^\circ$-submodule of $M$. The converse holds when every submodule  of $M$ is of the form $(0:_Mr)$ for some $r \in R$.
\end{rem}

\begin{prop}\label{pr551.1}
Let $N$ be a submodule of a cocyclic $R$-module $M$. If $N$ is a $dz^\circ$-submodule of $M$, then $N$ is a $dqz^\circ$-submodule of $M$. The converse holds when $M$ is a comultiplication $R$-module.
\end{prop}
\begin{proof}
Let $a \in Ann_R(N)$. Then $N \subseteq (0:_Ma)$. Since $M$ is a cocyclic module, the $(0)$ is a completely irreducible submodule of $M$. Thus by ,  $(0:_Ma)$ is a completely irreducible submodule of $M$.  Thus by assumption, $N \subseteq \mathfrak{S}_{(0:_Ma)}$, as needed. For converse, let $M$ be a comultiplication $R$-module and $L$ be a completely irreducible submodule of $M$ such that $N \subseteq L$. Then $L=\cap _{a \in Ann_R(L)}(0:_Ma)$ and so $L=(0:_Ma)$ for some $a \in Ann_R(L)$. It follows that  $a \in Ann_R(N)$.
Now the result is clear.
\end{proof}

\begin{lem}\label{l1.1}
Let $M$ be an $R$-module.
A submodule $N$ of $M$ is a $dqz^\circ$-submodule if and only if $N=\cap_{a\in Ann_R(N)} \mathfrak{S}_{(0:_Ma)}$.
\end{lem}
\begin{proof}
This is clear.
\end{proof}

\begin{prop}\label{p1000.1}
Let $N$ be a non-zero submodule of an $R$-module $M$. Then $N$ as an $R$-submodule is a $dqz^\circ$-submodule if and only if as an $R/Ann_R(M)$-submodule is a $dqz^\circ$-submodule.
\end{prop}
\begin{proof}
This is clear.
\end{proof}

\begin{lem}\label{t19.3}
Let $M$ be a faithful finitely generated comultiplication $R$-module. Then $N$ is a $dqz^\circ$-submodule of $M$ if and only if $Ann_R(N)$ is a $z^\circ$-ideal of $R$.
\end{lem}
\begin{proof}
This follows from Theorem \cite[Theorem 2.3]{F407}.
\end{proof}

\begin{thm}\label{c19.3}
Let $N$ be a $dqz^\circ$-submodule of a coreduced faithful finitely generated comultiplication $R$-module $M$. Then $(K:_RN)$ is a  $z^\circ$-ideal of $R$ for each submodule $K$ of $M$. In particular, if $\mathfrak{S}_M=M$, then $(K:_RM)$ is a $z^\circ$-ideal of $R$ for each submodule $K$ of $M$.
 \end{thm}
\begin{proof}
As $N$ is a $dqz^\circ$-submodule, $Ann_R(N)$ is a $z^\circ$-ideal of $R$ by Lemma \ref{t19.3}. Let $K$ be a submodule of $M$. Then  by \cite[Examples of $z^\circ$-ideals]{MR1736781}, $(Ann_R(N):_RAnn_R(K))$ is a $z^\circ$-ideal of $R$. Now
$$
(K:_RN)=((0:_MAnn_R(K)):_RN)=(Ann_R(N):_RAnn_R(K))
$$
implies that $(K:_RN)$ is a $z^\circ$-ideal of $R$. Now the last assertion follows from the fact that $\mathfrak{S}_M$ is a $dqz^\circ$-submodule of $M$ by Remark \ref{r1.1}.
\end{proof}

\begin{prop}\label{p1.1}
Let $M$ be a faithful finitely generated comultiplication $R$-module. If $N$ is a $dqz^\circ$-submodule of $ M$, then $Ann_R(N) \subseteq W_R(M)$.
\end{prop}
\begin{proof}
By \cite[Lemma 2.7]{MR3755273}, $Zd_R(R)=W_R(M)$. Now the result follows from the fact that $Ann_R(N)$ is a $z^\circ$-ideal of $R$ by Lemma \ref{t19.3}.
\end{proof}

The following theorem gives some characterizations for $dqz^\circ$-submodules of a faithful finitely generated coreduced comultiplication $R$-module $M$.
\begin{thm}\label{t1.6}
Let $M$ be a faithful finitely generated coreduced comultiplication $R$-module. Then the following are equivalent:
\begin{itemize}
\item [(a)] $N$ is a $dqz^\circ$-submodule of $M$;
\item [(b)] For each $a \in Ann_R(N)$ and submodule $K$ of $M$, $\mathfrak{S}_{(0:_Ma)}=\mathfrak{S}_{K}$ implies that $N \subseteq K$;
\item [(c)] For each $a \in Ann_R(N)$ and submodule $K$ of $M$, $V^s((0:_Ma))=V^s(K)$ implies that $N \subseteq K$;
\item [(d)] For each $a \in R$, we have $a\in Ann_R(N)$ implies that $N \subseteq Ann_R(aM)M$;
\item [(e)] For $a, b \in R$, $Ann_R(aM)=Ann_R(bM)$ and $a\in Ann_R(N)$ imply that $b\in Ann_R(N)$;
\item [(f)] For $a, b\in R$, $Ann_R(aM)\subseteq Ann_R(bM)$ and $a\in Ann_R(N)$ imply that $b\in Ann_R(N)$.
\end{itemize}
\end{thm}
\begin{proof}
$(a)\Rightarrow (b)$
Let $a \in Ann_R(N)$ and $K$ be a submodule of $M$ such that $\mathfrak{S}_{(0:_Ma)}=\mathfrak{S}_{K}$. By part (a), $N \subseteq \mathfrak{S}_{(0:_Ma)}$. Thus $N \subseteq \mathfrak{S}_{K} \subseteq K$.

$(b)\Rightarrow (c)$
Let $a \in Ann_R(N)$ and $K$ be a submodule of $M$ such that $V^s((0:_Ma))=V^s(K)$. Then  $\mathfrak{S}_{(0:_Ma)}=\mathfrak{S}_{K}$. Thus by part (b), $N \subseteq K$.

$(c)\Rightarrow (d)$
Let $a\in Ann_R(N)$. Then  $V^s((0:_Ma))=V^s(Ann_R(aM)M)$  by using \cite[Theorem 2.5(a)]{F407}. Thus by part (c), $N \subseteq Ann_R(aM)M$.

$(d)\Rightarrow (e)$
Let $a, b  \in R$, $Ann_R(aM)=Ann_R(bM)$ and  $a \in Ann_R(N)$. Then $Ann_R(aM)M=Ann_R(bM)M$. By part (d), $N\subseteq Ann_R(aM)M$. Thus $N\subseteq Ann_R(bM)M\subseteq (0:_Mb)$. Hence $b\in Ann_R(N)$.

$(e)\Rightarrow (f)$
By \cite[Corollary 2.7]{F407}.

$(f)\Rightarrow (a)$
Let $a\in Ann_R(N)$. By \cite[Corollary 2.20]{F407}, $Ann_R(aM)M=\mathfrak{S}_{(0:_Ma)}$. Let $Ann_R(aM)M\subseteq L$, where $L$ is a completely irreducible submodule of $M$. Then $Ann_R(aM)\subseteq (L:_RM)$. As $M$ is a comultiplication $R$-module, $L=(0:_MJ)$ for some ideal $J$ or $R$. Let $b \in J$. Then $ Ann_R(JM) \subseteq Ann_R(bM)$. Therefore, $Ann_R(aM) \subseteq Ann_R(bM)$. Thus by part (f), $b \in Ann_R(N)$ and so $N\subseteq (0:_MJ)=L$. This implies that
$N\subseteq Ann_R(aM)M=\mathfrak{S}_{(0:_Ma)}$ by Remark \ref{r2.1}.
\end{proof}

\begin{cor}\label{c1.7}
Every non-zero submodule of a faithful finitely generated coreduced multiplication and comultiplication $R$-module is a $dqz^\circ$-submodule.
\end{cor}
\begin{proof}
Since $M$ is a multiplication $R$-module, for each $a \in M$ we have $(0:_Ma)=((0:_Ma):_RM)M=Ann_R(aM)M$. Let $N$ be a non-zero submodule of $M$ and $a \in Ann_R(N)$. Then $N \subseteq (0:_Ma)$.  Now the result follows from Theorem \ref{t1.6} $(d)\Rightarrow (a)$.
\end{proof}

\begin{ex}\label{e3.9}
The $\Bbb Z_n$-module $\Bbb Z_n$, where $n$ is square free, is a faithful finitely generated coreduced comultiplication and multiplication $\Bbb Z_n$-module. Thus
 each non-zero submodule of $\Bbb Z_n$-module $\Bbb Z_n$ is a $dqz^\circ$-submodule by Corollary \ref{c1.7}.
\end{ex}

\begin{prop}\label{pp0701.14}
Let $N$ be a $dqz^\circ$-submodule of an $R$-module $M$. Then
for each $r \in R$, $rN$ is a $dqz^\circ$-submodule of $M$.
\end{prop}
\begin{proof}
Suppose that $r \in R$ and $a \in Ann_R(rN)$. Then $ar \in Ann_R(N)$. Now by assumption, $N \subseteq \mathfrak{S}_{(0:_Mar)}$. By Lemma \ref{ll0701.14}, we have
$$
\mathfrak{S}_{(0:_Mar)}=\mathfrak{S}_{((0:_Ma):_Mr)}\subseteq (\mathfrak{S}_{(0:_Ma)}:_Mr).
$$
Thus $N\subseteq (\mathfrak{S}_{(0:_Ma)}:_Mr)$ and so $rN \subseteq \mathfrak{S}_{(0:_Ma)}$, as needed.
\end{proof}

\begin{thm}\label{t4.8}
Let $M$ be a faithful finitely generated comultiplication $R$-module. Then the following are equivalent:
\begin{itemize}
\item [(a)] $M$ is a coreduced module, i.e., $R$ is a reduced ring;
\item [(b)] The submodule $M$ is a $dqz^\circ$-submodule of $M$.
\end{itemize}
\end{thm}
\begin{proof}
$(a)\Rightarrow (b)$
Let $a \in Ann_R(M)$. Then $(0:_Ma)=M$ and so $((0:_Ma):_RM)M=M$.
Now, the result follows from Theorem \ref{t1.6} $(d)\Rightarrow (a)$.

$(b)\Rightarrow (a)$
Let $a \in R$ such that $(Ra)^2=0$. It is clear that $\mathfrak{S}_{(0:_Ma)}=\mathfrak{S}_{(0:_Ma^2)}$. Thus
$\mathfrak{S}_{(0:_Ma)}=\mathfrak{S}_{(0:_Ma^2)}=\mathfrak{S}_{M}$. Since the submodule $M$ is a $dqz^\circ$-submodule, $(0:_Ma)=M$ by Theorem \ref{t1.6} $(a)\Rightarrow (b)$. Now since $M$ is faithful, $a=0$ as required.
\end{proof}

\begin{thm}\label{t7.9}
Let $N$ be a $dqz^\circ$-submodule of a faithful finitely generated coreduced comultiplication $R$-module $M$. Then every
maximal second submodule contained in $N$ is a $dqz^\circ$-submodule of $M$.
\end{thm}
\begin{proof}
Let $S$ be a maximal second submodule contained in $N$. Assume that $Ann(aM) \subseteq Ann(bM)$, where
$a\in Ann_R(S)$ and $b \in R$. By \cite[Proposition 2.1]{F407}, there exists $c \in Ann_R(aM) \setminus Ann_R(S)$. So, $ca \in Ann_R(N)$ and $c \not \in Ann_R(S)$. We have  $Ann(caM) \subseteq Ann(cbM)$. As $N$ is a $dqz^\circ$-submodule of $M$, we get that
$cb \in Ann_R(N)\subseteq Ann_R(S)$. As $c \not \in Ann_R(S)$ and $Ann_R(S)$ is a prime ideal of $R$, $b \in Ann_R(S)$. Now the result follows from Theorem \ref{t1.6} $(f)\Rightarrow (a)$.
\end{proof}

\begin{cor}\label{c7.12}
If $M$ is a faithful finitely generated coreduced comultiplication $R$-module, then every minimal $dqz^\circ$-submodule is a
second $dqz^\circ$-submodule.
\end{cor}

\begin{cor}\label{c7.13}
Let $M$ be a faithful finitely generated coreduced comultiplication $R$-module and $S$ be a second submodule of $M$. Then either $S$ is a $dqz^\circ$-submodule or contained in a minimal $dqz^\circ$-submodule which is a second $dqz^\circ$-submodule.
\end{cor}

\begin{cor}\label{c7.10}
Let $f: N\hookrightarrow M$ be the natural inclusion, where $M$ is  a finitely generated faithful coreduced comultiplication $R$-module and $N$ is a $dqz^\circ$-submodule of $M$. If $K$ is a $dqz^\circ$-submodule of $N$, then $f(N)$ is $dqz^\circ$-submodule of $M$.
\end{cor}

\end{document}